\documentclass[12pt,reqno]{amsart}

\usepackage[utf8]{inputenc}

\usepackage{amsmath} 
\usepackage{amsthm} 
\usepackage{amssymb} 
\usepackage{amscd} 

\usepackage{euscript}
\DeclareMathAlphabet{\mathpzc}{OT1}{pzc}{m}{it}
\DeclareMathAlphabet\euscr{T1}{qzc}{m}{n}

\usepackage{emptypage}
\usepackage{enumitem} 
\usepackage{mathtools} 
\usepackage{mathrsfs} 
\usepackage{hyperref} 
\hypersetup{
	linktocpage = true,
	colorlinks=true,
	linkcolor=red,
}
\usepackage{esint} 

\usepackage{graphicx} 
\usepackage{subcaption}

\usepackage{pgf,tikz,pgfplots}
\usetikzlibrary{arrows}
\usetikzlibrary{intersections}
\usetikzlibrary{fadings}

\usepackage[colorinlistoftodos]{todonotes} 
\newtheorem{theorem}{Theorem}[section]
\newtheorem*{theorem*}{Theorem}
\newtheorem{proposition}[theorem]{Proposition}
\newtheorem{lemma}[theorem]{Lemma}
\newtheorem{corollary}[theorem]{Corollary}

\theoremstyle{definition}

\theoremstyle{remark}
\newtheorem{remark}[theorem]{Remark}
\newtheorem*{remark*}{Remark}

\newcommand{\con}[1]{\mathbb{#1}}
\newcommand{\R}{\con{R}} 
\newcommand{\Z}{\con{Z}} 

\makeatletter
\newcommand{\leqnomode}{\tagsleft@true\let\veqno\@@leqno}
\newcommand{\reqnomode}{\tagsleft@false\let\veqno\@@eqno}
\makeatother

\newcommand\beqc[1]{\left\{\begin{array}{#1}}
\newcommand\eeqc{\end{array} \right.}

\def\bmatrix{\begin{pmatrix}}
\def\ematrix{\end{pmatrix}}

\DeclareMathOperator{\PV}{P.V.}
\let\div\relax
\DeclareMathOperator{\div}{div}

\numberwithin{equation}{section}

\graphicspath {{Images/}}

\title[Uniqueness for integro-differential equations in 1D and applications]
{Uniqueness for linear integro-differential equations in the real line and applications}

\author{Juan-Carlos Felipe-Navarro}
\address{J.C. Felipe-Navarro:
	Universitat Polit\`ecnica de Catalunya and BGSMath, Departament de Matem\`{a}tiques, Diagonal 647, 08028 Barcelona, Spain}
\email{juan.carlos.felipe@upc.edu}

\thanks{The author acknowledges financial support from the Spanish Ministry of Economy and Competitiveness (MINECO), through the Mar\'ia de Maeztu Programme for Units of Excellence in R\&D (MDM-2014-0445-16-4), is supported by MINECO grant MTM2017-84214-C2-1-P, is member of the Barcelona Graduate School of Mathematics (BGSMath), and is part of the Catalan research group 2017 SGR 01392.}

\keywords{integro-differential operators, uniqueness, Liouville-type approach, nondegeneracy.}

\begin{document}
	
\begin{abstract}
In this work we prove the uniqueness of solutions to the nonlocal linear equation $L \varphi- c(x)\varphi = 0$ in $\R$, where $L$ is an elliptic integro-differential operator, in the presence of a positive solution or of an odd solution vanishing only at zero. As an application, we deduce the nondegeneracy of layer solutions (bounded and monotone solutions) to the semilinear problem $L u = f(u)$ in $\R$ when the nonlinearity is of Allen-Cahn  type. To our knowledge, this is the first work where such uniqueness and nondegeneracy results are proven in the nonlocal framework when the Caffarelli-Silvestre extension technique is not available. Our proofs are based on a nonlocal Liouville-type method developed by Hamel, Ros-Oton, Sire, and Valdinoci for nonlinear problems in dimension two.
\end{abstract}

\maketitle


\section{Introduction and main results}
\label{Sec:Introduction}

In this paper we study the uniqueness, up to a multiplicative constant, of solutions to the linear integro-differential equation
\begin{equation} \label{Ecuacion Lineal}
L\varphi-c(x)\varphi=0 \,\,\,\,\,\,\, \text{in} \,\,\R,
\end{equation}
under certain assumptions on the nonlocal operator $L$ and the potential function $c$, and in the presence of a positive solution or of an odd solution vanishing only at zero. Throughout the paper, $L$ will be assumed to be an elliptic integro-differential operator of order between one (included) and two.

The uniqueness of solutions to equations of the form \eqref{Ecuacion Lineal} is a very important tool in the theory of PDEs. Indeed, there are many motivations (from both linear and nonlinear frameworks) to treat this problem. On the one hand, it is in the essence of Sturm--Liouville theory on eigenfunctions and eigenvalues. On the other hand, it has important consequences when studying qualitative properties of solutions to semilinear problems. For instance, in the context of nonlinear Schrödinger equations, the nondegeneracy of ground state solutions (which plays a very important role in the stability and blow up analysis of solitary waves to related time-dependent equations) is reduced to study the uniqueness of solution to equation \eqref{Ecuacion Lineal} when $L$ is replaced by the radial component of the Laplacian, i.e., $L=r^{1-n} (r^{n-1} u_r)_r$  (see \cite{Gesztesy-Schrodinger}). Furthermore, in the framework of the Allen-Cahn equation, Berestycki, Caffarelli, and Nirenberg \cite{BerestyckiCaffarelliNiremberg-Qualitative} realized that the uniqueness of solutions to equation \eqref{Ecuacion Lineal} in dimension $n$ for the local case (with $L$ replaced by a general second order uniformly elliptic operator) leads to the resolution of a conjecture by De Giorgi for monotone solutions.\footnote{All the directional derivatives of a solution $u$ to the Allen-Cahn equation satisfy the linearized problem. Moreover, from the monotonicity assumption in the De Giorgi conjecture we know that one of the derivatives is positive, say $\partial_n u >0$. Therefore, we can apply the uniqueness result to the linearized equation in the presence of a positive solution to conclude that all partial derivatives are multiples of $\partial_n u$. In particular, the gradient of $u$ has a fixed direction, which turns out to be equivalent to the one-dimensionality of the solution $u$.}

In the present paper, equation~\eqref{Ecuacion Lineal} is driven by a translation invariant integro-differential operator of the form
\begin{equation} \label{Tipo operador}
Lu(x) = \PV\,\int_{\R^n} \big(u(x)-u(y)\big)\,K(x-y)\,dy.
\end{equation}
In this nonlocal setting there are lots of basic open problems concerning solutions in dimension one, unlike the case of local equations where the one dimensional problem \eqref{Ecuacion Lineal} is just a second order linear ODE. For instance, a full understanding of the phase portrait of solutions in the nonlocal framework is missing.

Most of the works in the literature concerning uniqueness of solutions to \eqref{Ecuacion Lineal}-\eqref{Tipo operador} treat the simplest case $L=(-\Delta)^s$ (see \cite{FrankLenzmann, FrankLenzmannSilvestre, ChenD'AmbrosioLi, RosOtonSerra-Stable, Fall-Liouville, CabreSolaMorales, CabreSireI} and the comments along this introduction). In such a scenario, the main analytic tools are potential theory, Fourier analysis, and the Caffarelli-Silvestre extension problem. Since they are not available when dealing with more general integro-differential operators, new techniques are needed. In \cite{HamelRosOtonSireValdinoci}, Hamel, Ros-Oton, Sire, and Valdinoci develop a purely nonlocal method (in contrast to the local extension problem) to treat these operators. They use it to establish a uniqueness result in dimension two (motivated by a nonlocal version of De Giorgi's conjecture) in the case of operators with compactly supported kernel and power-like behavior at the origin. In the present paper, their methodology is used in dimension one for the first time. It leads to uniqueness results for equations of the form \eqref{Ecuacion Lineal}-\eqref{Tipo operador}. Working in dimension one allows us to get rid of the compact support assumption in \cite{HamelRosOtonSireValdinoci}.

Throughout the paper, we assume that the kernel $K$ of the integro-differential operators satisfies the positivity and symmetry conditions
\begin{equation}
	\tag{K1}
	\label{eq: K1}
	K(z) >0 \ \ \ \ \text{and} \ \ \ \   K(-z) = K(z),
\end{equation}
together with an ellipticity assumption. That is, to be bounded both from above and below by a multiple of the kernel of the fractional Laplacian, i.e.,
\begin{equation}
	\tag{K2}
	\label{eq: K2}
	\frac{\lambda}{|z|^{n+2s}} \leq  K(z) \leq \frac{\Lambda}{|z|^{n+2s}},
\end{equation}
for some constants $\Lambda\geq \lambda > 0$ and $s\in[1/2,1)$. Note that the operator $L$ will be assumed to be of order between one (included) and two. Condition \eqref{eq: K2} is one of the most frequently adopted when dealing with nonlocal operators of the form \eqref{Tipo operador}. It is known to yield Hölder regularity of solutions (see
\cite{RosOton-Survey} and \cite{Serra-C2s+alphaRegularity}).

In some results the lower bound will not be assumed, and the upper one can be relaxed to
\begin{equation}
	\tag{K3}
	\label{eq: K3}
	K(z) \leq \frac{\Lambda_1}{|z|^{n+2\underline{s}}}+\frac{\Lambda_2}{|z|^{n+2\overline{s}}},
\end{equation}
for some constants $\Lambda_1,\Lambda_2\geq 0$ and $1/2\leq \underline{s}\leq \overline{s}<1$. This is the case of Theorem~\ref{Mi teorema general} and Corollary~\ref{Corolario suma fraccionarios}.

We will sometimes assume the potential function $c$ to be negative at infinity. That is,
\begin{equation}
	\label{eq: V1}
	c(x) \leq -c_0 < 0 \,\,\,\,\,\, \text{in} \,\,\, \R\setminus[-R_0,R_0],
\end{equation}
for some positive constants $c_0$ and $R_0$.

The following is our first important result. It establishes the uniqueness of solution to \eqref{Ecuacion Lineal} in the presence of a positive one (in addition to other assumptions).

\begin{theorem}{} \label{Teorema L0}
	Let $L$ be an integro-differential operator of the form \eqref{Tipo operador} satisfying the symmetry and ellipticity conditions \eqref{eq: K1} and \eqref{eq: K2} with $s \in [1/2,1)$. For $\alpha>2s-1$, let $w$ and $\widetilde{w}$ be two $C^{1,\alpha}$ solutions of the linear equation
	$$ L\varphi-c(x)\varphi=0 \,\,\,\,\,\,\, \text{in} \,\,\,\R, $$
	with 
	$$w>0.$$
	Assume that 
	\begin{itemize}[leftmargin=*]
		\item[$\bullet$] either both $w$ and $\tilde{w}$ are bounded and the potential function $c\in L^\infty(\R)$ satisfies
		$$ c(x) \leq -c_0 < 0 \,\,\,\,\,\, \text{in} \,\,\, \R\setminus[-R_0,R_0], \,\,\,\,\,\, \text{ and } \,\,\,\,\,\, ||c||_{C^{\beta_0}(\R)} < +\infty $$
		for some positive constants $c_0$, $R_0$, and $\beta_0$;
		\item[$\bullet$]  or $w$ satisfies
		$$ 0 < C^{-1} \leq w(x) \leq C\,\,\,\,\,\,\, \text{in} \,\,\,\R   $$
		and $\tilde{w}$ is such that
		$$ ||\widetilde{w}||_{L^\infty(-R,R)} \leq C R^{s-\frac{1}{2}} \, \, \text{ for all } \, \, R>1,$$
		for some positive constant $C$.
	\end{itemize}
	
	Then
	$$ \frac{\widetilde{w}}{w} \equiv \text{constant}. \footnote{The result can also be established (see the proof of Theorem~\ref{Mi teorema general} and the estimates in Section~\ref{sec:Integrability}) in the second scenario for $s\in (0,1/2)$ if one assumes that the solution $\tilde{w}$ is bounded and decays as $||\widetilde{w}||_{L^\infty(\R\setminus(-R,R))} \leq C R^{s-\frac{1}{2}}$ for $R>1$. In this precise case one would conclude that $\tilde{w} \equiv 0$ since $w>C^{-1}$ in $\R$.} $$
\end{theorem}

Let us point out that some assumptions concerning the kernel can be relaxed to include a bigger class of operators (see Theorem~\ref{Mi teorema general} for the precise statement) such as the sum of fractional Laplacians with different order (see Corollary~\ref{Corolario suma fraccionarios}). Nevertheless, for the sake of clarity and simplicity we prefer to state Theorem~\ref{Teorema L0} here.

To the best of our knowledge, Theorem~\ref{Teorema L0} is the first uniqueness result for general integro-differential operators in dimension one. Previous analogue results could only cover the case of the fractional Laplacian (see Remark~\ref{Remark:CommentsMainTheorem} for comments on such works). 

In order to prove uniqueness we follow a Liouville-type method. The main idea consists of finding an equation for the quotient of two solutions, which is the crucial contribution by Hamel, Ros-Oton, Sire, and Valdinoci \cite{HamelRosOtonSireValdinoci} for general integro-differential operators, and then showing that such a quotient is constant. This requires a growth estimate in both the local and nonlocal cases.

Unlike \cite{HamelRosOtonSireValdinoci}, where a key point is assuming that the kernels have compact support, we adapt the strategy in order to remove such a condition by taking advantage of the one dimensionality of the problem. In our approach, the first step is controlling the growth of the quotient of the solutions. This comes for free when the positive solution is just bounded from below by a strictly positive constant. However, a finer analysis is needed when the positive solution can be arbitrarily close to zero at infinity. In that case, we prove the boundedness of the quotient by using condition $\eqref{eq: V1}$ and the boundedness of the solutions. Here, we use a maximum principle in the exterior of an interval, proved in Section~\ref{sec:MaxPrinciple}, in order to compare both solutions by transferring the information from the interval (where we know the quotient is bounded) to the whole line. The second ingredient to prove the uniqueness theorem is an integral estimate for the function $K(x-y)$ with respect to both variables $x$ and $y$ in unbounded cross-shaped regions of the plane. In fact, the validity of this estimate is what prevents us from extending our result to $s\in(0,1/2)$. We show it in Section~\ref{sec:Integrability}. Let us point out that both ingredients become trivial when working with kernels with compact support, as it is done in \cite{HamelRosOtonSireValdinoci}.

\begin{remark} \label{Remark:CommentsMainTheorem}
	As it is natural, our result, which includes a big  class of integro-differential operators, is not optimal when we apply it to the fractional Laplacian. In order to compare it with other similar results in the literature, let us distinguish two cases depending on whether the equation has a zeroth order term or not.
	
	On the one hand, when $c\equiv 0$, in \cite{BogdanKulczyckiNowak-Liouville}, Bogdan, Kulczycki, and Nowak used a gradient estimate to show that nonnegative $s$-harmonic functions are constant. Later on, Chen, D'Ambrosio, and Lin \cite{ChenD'AmbrosioLi} proved, by using potential theory and Fourier analysis, a Liouville theorem for the fractional Laplacian with the growth condition 
	$$\liminf_{|x|\to \infty} \frac{u(x)}{|x|^\gamma} = 0, $$
	if $0\leq \gamma\leq 1$ and $\gamma<2s$. In this scenario, our result, by taking $w\equiv 1$ as the positive solution, leads to solutions growing less or equal than $|x|^{s-1/2}$ at infinity being constant. Thus, we notice what we have previously announced, that our condition is not sharp for the fractional Laplacian.
	
	On the other hand, when the potential function is not identically zero, it is known that the uniqueness result for the fractional Laplacian, with $s \in [1/2,1)$, follows from Theorem~4.10 by Cabré and Sire in \cite{CabreSireI} (see also the work by Cabré and Solà-Morales \cite{CabreSolaMorales} for the half-Laplacian) and the use of the local extension problem. In this case, unlike our result, no condition on the potential function (or the positive solution) needs to be assumed.
\end{remark}

An important and direct application of Theorem~\ref{Teorema L0} is the nondegeneracy of layer solutions to Allen-Cahn type equations. Let us recall that a bounded solution to the semilinear problem 
\begin{equation} \label{eq:Semilinear}
	L u = f(u) \,\,\,\,\,\,\, \text{in} \,\,\R,
\end{equation}
is called \emph{layer solution} if it is strictly increasing. In particular, it has limits at infinity, which (without loss of generality) we can consider to be $\pm 1$. 

When $L$ is a second order differential operator, layer solutions to equation \eqref{eq:Semilinear} are just particular cases of heteroclinic connections to nonlinear ODEs. Nevertheless, in the nonlocal setting, even the existence of such solutions is not an easy problem due to the lack of an analogous nonlocal ODE theory. In the fractional case $L= (-\Delta)^s$, existence and uniqueness are shown in \cite{CabreSolaMorales, CabreSireI, CabreSireII} by using the extension problem. For more general integro-differential operators, we can refer to the work by Cozzi and Passalacqua \cite{CozziPassalacqua} where they prove existence, uniqueness (up to translations), and some qualitative properties of layer solutions (see \cite{Felipe-NavarroSanz-Perela:IntegroDifferentialI} for further properties). Here, we prove nondegeneracy:

\begin{theorem} \label{Nondegeneracy Layer}
	Let $L$ be an integro-differential operator of the form \eqref{Tipo operador} satisfying the symmetry and ellipticity conditions \eqref{eq: K1} and \eqref{eq: K2}  with $s \in [1/2,1)$. For $\gamma>0$, let $f\in C^{1,\gamma}([-1,1])$ be any given nonlinearity such that $f'(\pm1)<0$. 
	
	Assume that $u$ is a bounded solution to the semilinear equation \eqref{eq:Semilinear}, satisfying $u'>0$ and $\lim_{x\to\pm\infty} u(x) = \pm 1$. 
	
	Then, $u$ is nondegenerate, i.e., up to a multiplicative constant, $u'$ is the unique bounded solution to the linearized equation $L \varphi - f'(u) \varphi = 0$ in $\R$.
\end{theorem}

Let us point out that condition $f'(\pm 1) <0$, which corresponds to $c=f'(u)$ being negative at infinity, is a natural assumption. Indeed, it is the same hypothesis needed to prove uniqueness (up to translations) of the layer solutions (see Theorem~1.2 in \cite{CabreSolaMorales} in the case of the half-Laplacian). Moreover, this is also the needed condition for $\pm 1$ to be local minimizers of the associated energy.

The nondegeneracy of solutions plays a very relevant role in the stability and blow up analysis for time dependent versions of equation \eqref{Ecuacion Lineal}. Furthermore, it is also important in stationary problems, as in the construction of new solutions to the semilinear equation \eqref{eq:Semilinear} around a nondegenerate one by using an implicit function argument. Indeed, Dávila, del Pino, and Musso \cite{DaviladelPinoMusso} proved the nondegeneracy of the layer solution when $L= (-\Delta)^{1/2}$ (with the extension problem) in order to construct solutions to \eqref{eq:Semilinear} that develop multiple transitions from $-1$ to $1$. In \cite{DuGuiSireWei}, Du, Gui, Sire, and Wei generalize the nondegeneracy to $s\in (1/2,1)$ and use it to show the existence of clustering-layered solutions for a fractional inhomogeneous Allen-Cahn equation.

Next, we present the third main result of this work: a uniqueness theorem in the odd setting. Let us point out that in such a case our strategy allows us to show uniqueness only among odd functions. Completely different arguments would be needed to establish uniqueness among all functions, as it occurs in \cite{FrankLenzmann} for a particular case involving the fractional Laplacian (see the end of the present introduction for more details).

\begin{theorem}{} \label{Mi teorema general impar}
Let $L$ be an integro-differential operator of the form \eqref{Tipo operador} with kernel $K$ being decreasing in $(0,+\infty)$ and satisfying the symmetry and ellipticity conditions \eqref{eq: K1} and \eqref{eq: K2} for some $s \in [1/2,1)$. Assume the potential function $c\in L^\infty(\R)$ satisfies $$ c(x) \leq -c_0 < 0 \,\,\,\,\,\, \text{in} \,\,\, \R\setminus[-R_0,R_0], \,\,\,\,\,\, \text{ and } \,\,\,\,\,\, ||c||_{C^{\beta_0}(\R)} < +\infty $$
for some positive constants $c_0$, $R_0$ and $\beta_0$.

For $\alpha>2s-1$, let $w$ and $\widetilde{w}$ be two odd $C^{1,\alpha}$ bounded solutions of the linear equation
$$ L\varphi-c(x)\varphi=0 \,\,\,\,\,\,\, \text{in} \,\,\,\R, $$
with 
$$w>0 \text{ in } (0,+\infty).$$

Then
$$ \frac{\widetilde{w}}{w} \equiv \text{constant}. $$
\end{theorem}

Note that since the integro-differential operator $L$ preserves the oddness of functions, the potential function $c$ needs to be even if we assume the problem to have existence of odd solutions. On the other hand, the monotonicity of the kernel is a natural assumption when working with odd functions in the nonlocal setting. Indeed, for the validity of the maximum principle (see Lemma~\ref{Lemma: AlternativeExpressionOdd} and section 3 of \cite{JarohsWeth}), this condition is the analogue in the odd framework to the positivity of the kernel in \eqref{eq: K1} for general functions.

As in Theorem~\ref{Nondegeneracy Layer} for the case of functions without any symmetry, we can apply the previous uniqueness result to prove qualitative properties of solutions to semilinear problems. Let us recall that a bounded solution (without loss of generality we can consider it to be bounded by $1$) to the semilinear equation \eqref{eq:Semilinear} is called \emph{ground state} if it is even, positive, and decreasing to zero at infinity. We refer to the work by Frank and Lenzmann \cite{FrankLenzmann} and references therein for existence results of such solutions. Here, we establish a partial nondegeneracy result (in the sense that we prove uniqueness for the linearized equation only among odd functions):

\begin{theorem} \label{Nondegeneracy Ground State}
Let $L$ be an integro-differential operator of the form \eqref{Tipo operador} with kernel $K$ being decreasing in $(0,+\infty)$ and satisfying the symmetry and ellipticity conditions \eqref{eq: K1} and \eqref{eq: K2}, for some $s \in [1/2,1)$. For $\gamma>0$, let $f\in C^{1,\gamma}([0,1])$ be any given nonlinearity such that $f'(0)<0$. 

Assume that $u$ is a bounded even solution to the semilinear equation \eqref{eq:Semilinear}, satisfying $u'<0$ in $(0,+\infty)$ and $\lim_{x\to\pm\infty} u(x) = 0$. 

Then, up to a multiplicative constant, $u'$ is the unique bounded odd solution to the linearized equation $L \varphi - f'(u) \varphi = 0$ in $\R$.
\end{theorem}

As in the nondegeneracy result for layer solutions, the condition $f'(0) <0$ is a natural assumption. Indeed, it is a necessary condition in order for $v\equiv 0$ to be a local minimizer of the associated energy.

The most important result in the literature dealing with nondegeneracy of ground states in the nonlocal framework is due to Frank and Lenzmann \cite{FrankLenzmann}. Unlike us, they were able to establish the full nondegeneracy (uniqueness for the linearized equation among all functions) in the particular case $L=(-\Delta)^s$ and $f$ being a polynomial nonlinearity (see Lemma~C.3 from \cite{FrankLenzmann}) as we explain next. An important point in their strategy is to note that the operator $L-f'(u)$ preserves odd/even symmetry. Thus, both the odd and even parts of any given solution of the linearized problem are also solutions, and a separated analysis can be done for each one. First, they prove the uniqueness among odd functions by using the heat kernel for the fractional Laplacian. Next, they show that the unique even solution is the trivial one, which is the most difficult step. In order to do it, they develop a delicate spectral theory for fractional Schrödinger operators (where the local extension problem and the polynomial structure of the nonlinearity play a crucial role).  Finally, the uniqueness among all functions follows from the previous results. The nondegeneracy of ground states turns out to be very important since they use it to prove their uniqueness result by using an implicit function argument and the well known result for the local case ($s=1$).

Finally, let us comment that the strategy to prove Theorem~\ref{Mi teorema general impar}  follows the same lines as the one of Theorem~\ref{Teorema L0}. Nevertheless, there are some difficulties we have to overcome. First, we need to take advantage of the odd symmetry to find an equation for the quotient of two solutions (see Corollary \ref{Lema ecuacion sigma impar}) which involves only the values of the functions in $(0,\infty)$, where the first solution $w$ is known to be positive. Next, we need to assure the quotient to be well-defined at the origin, where the denominator vanishes. We can accomplish it by using a maximum principle in small domains around the origin and taking into account that the numerator also vanishes at this point.

The paper is organized as follows. In Section~\ref{sec:Preliminary} we present the equation satisfied by the quotient of two solutions to \eqref{Ecuacion Lineal}-\eqref{Tipo operador}. Section~\ref{sec:MaxPrinciple} is devoted to show the maximum principles in the exterior of an interval. In Section~\ref{sec:Integrability} we give some estimates involving the integral of the kernel in cross-shaped unbounded domains. Finally, in Sections~\ref{sec:ProofMainResults} and \ref{sec:ProofMainResultsOdd} we prove the main results of the paper.


\section{Preliminary results: An equation for the quotient of solutions} \label{sec:Preliminary}
In this section we include a few preliminary algebraic computations that will be employed in the
proof of the main theorems. They are inspired by the computations done by Hamel, Ros-Oton, Sire, and Valdinoci in \cite{HamelRosOtonSireValdinoci}.

In the local framework (see proof of Theorem~1.8 in Section~4 of \cite{BerestyckiCaffarelliNiremberg-Qualitative}), it is well known that given a positive supersolution $w$ and a solution $\tilde{w}$ to the linear equation $-\Delta \varphi - c(x) \varphi = 0$, the quotient $\sigma := \tilde{w}/w$ satisfies $\sigma \div(w^2\nabla \sigma) \geq 0$. Thus, multiplying by $\tau^2$, where $\tau$ is any cut-off function, and integrating in the whole space, one arrives at
\begin{equation} \label{Eq:LocalQuotient}
	2\int_{\R^n} \tau^2(x) w^2(x) |\nabla \sigma(x)|^2 dx \leq - \int_{\R^n} w^2(x) \nabla (\tau^2(x)) \cdot \nabla (\sigma^2(x)) dx.
\end{equation}
Similar computations can also be done, by using the extension problem, when the Laplacian is replaced by the fractional Laplacian (see \cite{CabreSolaMorales, CabreSireI}).

In the general integro-differential case we establish the following:

\begin{lemma}\label{Lema ecuacion sigma}
	Let $L$ be an integro-differential operator of the form \eqref{Tipo operador}. Assume that $w$ and $\sigma$ are two smooth functions such that $w$ and $\tilde{w}:=\sigma w$ satisfy
	$$ w\left(Lw-cw\right)\geq 0 \,\,\,\,\,\,\, \text{in} \,\,\R $$
	and
	$$ \tilde{w}\left(L\tilde{w}-c\tilde{w}\right) \leq 0 \,\,\,\,\,\,\, \text{in} \,\,\R, $$
	respectively, for some potential function $c=c(x)$.
	
	Then, given any function $\tau\in C_c^\infty(\R)$,
	\begin{align*}
	\int_\R\int_\R \big( \sigma(x)-\sigma(y) \big)^2&\big(\tau^2(x)+\tau^2(y) \big)\,w(x)\,w(y)\,K(x-y)\,dx dy \\ &\hspace{-20mm}\leq -
	\int_\R\int_\R \big( \sigma^2(x)-\sigma^2(y) \big)\big( \tau^2(x)-\tau^2(y)\big)w(x)\,w(y)\,K(x-y)\,dx dy.
	\end{align*}
	Moreover, if $w\left(Lw-cw\right) = \tilde{w}\left(L\tilde{w}-c\tilde{w}\right) = 0$, equality holds in the previous expression.
\end{lemma}

This result, which is a generalization of Lemma 2.1 from \cite{HamelRosOtonSireValdinoci}, is a nonlocal analogue to \eqref{Eq:LocalQuotient}. In Section~\ref{sec:ProofMainResults}, we will use it to prove that the quotient of two solutions to the linear equation \eqref{Ecuacion Lineal} is constant.

\begin{proof}
	First, combining $w (Lw-cw) \geq 0$ and  $\tilde{w} (L\tilde{w}-c\tilde{w}) \leq 0$, we can easily check that $\sigma (\tilde{w} Lw - w L\tilde{w}) \geq 0$. Then, multiplying by $\tau^2$, where $\tau$ is any cut-off function, and repeating the algebraic computations done in \cite{HamelRosOtonSireValdinoci} we find that
	\begin{align*}
	\int_\R\int_\R \big( \sigma(x)-\sigma(y) \big)^2&\tau^2(x)\,w(x)\,w(y)\,K(x-y)\,dx dy \\ &\hspace{-20mm}\leq -
	\int_\R\int_\R \big( \sigma(x)-\sigma(y) \big)\big( \tau^2(x)-\tau^2(y)\big)\sigma(x)\,w(x)\,w(y)\,K(x-y)\,dx dy.
\end{align*}
	Finally, symmetrizing in both $x$ and $y$ we conclude the proof.
\end{proof}

As a consequence of the previous lemma, we can also find a useful identity for the quotient of two odd solutions to the linear equation \eqref{Ecuacion Lineal}. In such a case, all the integrals can be written in $(0,+\infty)$ by taking advantage of the symmetry of the functions.

\begin{corollary} \label{Lema ecuacion sigma impar}
	Let $L$ be an integro-differential operator of the form \eqref{Tipo operador}. Assume that $w$ and $\sigma$ are two smooth functions such that both $w$ and $\tilde{w}:=\sigma w$ are odd solutions to the linear equation
	$$ L\varphi-c(x)\varphi=0, \,\,\,\,\,\,\, \text{in} \,\,\R, $$
	for some even potential function $c=c(x)$. 
	
	Then, given any even function $\tau\in C_c^\infty(\R)$,
	\begin{align*}
	\int_0^\infty\int_0^\infty \big( \sigma(x)-\sigma(y) \big)^2&\big(\tau^2(x)+\tau^2(y) \big)\,w(x)\,w(y)\,\big\{K(x-y)-K(x+y) \big\}\,dx dy \\ &\hspace{-43mm}= -
	\int_0^\infty\int_0^\infty \big( \sigma^2(x)-\sigma^2(y) \big)\big( \tau^2(x)-\tau^2(y)\big)w(x)\,w(y)\,\big\{K(x-y)-K(x+y) \big\}\,dx dy.
	\end{align*}
\end{corollary}

Note that the previous identity is completely identical to the one in the general case but with integrals now computed in the half-line instead of the whole line, and with $K(x-y)-K(x+y)$ taking the role of $K(x-y)$.

\begin{proof}[Proof of Corollary~\ref{Lema ecuacion sigma impar}]
	We use the symmetry properties of the functions ($\sigma$ and $\tau$ are even while $w$ is odd) to rewrite the identity from Lemma \ref{Lema ecuacion sigma} in terms of integrals computed only in $\R^+$. That is,
	
	\begin{align*}
	\int_\R\int_\R &\big( \sigma(x)-\sigma(y) \big)^2\big(\tau^2(x)+\tau^2(y) \big)\,w(x)\,w(y)\,K(x-y)\,dx dy \\
	&=\int_\R\int_0^\infty \big( \sigma(x)-\sigma(y) \big)^2\big(\tau^2(x)+\tau^2(y) \big)\,w(x)\,w(y)\,\left[K(x-y)-K(x+y)\right]\,dx dy\\
	&= 2 \int_0^\infty\int_0^\infty \big( \sigma(x)-\sigma(y) \big)^2\big(\tau^2(x)+\tau^2(y) \big)\,w(x)\,w(y)\\
	&\hspace{80mm} \cdot\left[K(x-y)-K(x+y)\right]\,dx dy
	\end{align*}
	and
	\begin{align*}
	\int_\R\int_\R &\big( \sigma^2(x)-\sigma^2(y) \big)\big(\tau^2(x)-\tau^2(y) \big)\,w(x)\,w(y)\,K(x-y)\,dx dy \\
	&=\int_\R\int_0^\infty \big( \sigma^2(x)-\sigma^2(y) \big)\big(\tau^2(x)-\tau^2(y) \big)\,w(x)\,w(y)\,\left[K(x-y)-K(x+y)\right]\,dx dy\\
	&= 2 \int_0^\infty\int_0^\infty \big( \sigma^2(x)-\sigma^2(y) \big)\big(\tau^2(x)-\tau^2(y) \big)\,w(x)\,w(y)\\
	&\hspace{80mm} \cdot\left[K(x-y)-K(x+y)\right]\,dx dy.
	\end{align*}
	From this, we conclude the desired result by applying Lemma \ref{Lema ecuacion sigma}.
\end{proof}


\section{Some maximum principles in the exterior of an interval} \label{sec:MaxPrinciple}

In this section we prove two maximum principles in the exterior of an interval for some linear equations driven by an integro-differential operator plus a zeroth order term. The first result applies to functions without any symmetry, while the second one concerns odd functions. They will be the fundamental tool in Section~\ref{sec:ProofMainResults} and \ref{sec:ProofMainResultsOdd} to show that the quotient of two bounded solutions to equation \eqref{Ecuacion Lineal} is also bounded.

\begin{proposition}{} \label{Lema principio maximo}
	Let $L$ be an integro-differential operator of the form \eqref{Tipo operador} satisfying conditions \eqref{eq: K1} and \eqref{eq: K3} for some $1/2\leq \underline{s} \leq \overline{s}<1$. Assume that the potential function $c=c(x)$ satisfies \eqref{eq: V1} for some positive constants $R_0$ and $c_0$. 
	
	For $\alpha> 2\overline{s}-1$, let $\varphi$ be a bounded and $C^1$ function in $\R$ such that $[\varphi']_{C^{\alpha}(\R)} < +\infty$, 
	$$L\varphi-c\varphi \geq 0 \,\,\,\,\text{in } \,\R\setminus[-R_0,R_0],$$
	and
	$$\varphi \geq 0 \,\,\,\,\text{in } \,[-R_0,R_0].$$
	
	Then
	$$ \varphi \geq 0 \,\,\,\,\text{in } \,\R. $$
\end{proposition}

For simplicity, we are assuming $1/2\leq \underline{s} \leq \overline{s}<1$ since this is the range in which we are applying the result. However, the proof can be easily adapted to $0< \underline{s} \leq \overline{s}<1$ and any dimension (with the ball taking the role of the interval). Moreover, we point out that the negativity of the potential function $c$ at infinity, which is an assumption in some parts of Theorem~\ref{Teorema L0}, originates on this maximum principle.

\begin{proof}[Proof of Proposition \ref{Lema principio maximo}]
	Assume the result to be false. Then, the infimum of $\varphi$ is negative. In the case it is achieved, the contradiction comes directly from evaluating the operator $L\varphi-c\varphi$ at a point where such a minimum is attained. On the contrary, if the infimum is not achieved, we can construct a sequence of points $x_k \not\in [-R_0,R_0]$ where $\varphi$ takes negative values and approaches the infimum in the following way:
	\begin{align} \label{ecuacion principio maximo 3}
		\varphi(x_k) -  \varphi(x) \leq\varphi(x_k) - \inf_\R\,\varphi \leq \frac{1}{k}\,\,\, \text{for all } \,x\in\R.
	\end{align}
		
	Next, we evaluate $L\varphi-c\varphi$ at that sequence of points. In order to do it,
	we split the integro-differential term of the operator into two parts, and we estimate
	each one separately. That is,
	\begin{align*}
		L\,\varphi(x_k) &= \int_{-\infty}^{\infty} \big(\varphi(x_k) -  \varphi(y)\big) K(x_k-y) \,dy = \int_{-\infty}^\infty \big(\varphi(x_k) -  \varphi(x_k-z)\big) K(z) \,dz \\	
		&= \int_{\delta}^{\infty} \big(2\varphi(x_k) -  \varphi(x_k-z)-\varphi(x_k+z)\big) K(z) \,dz  \\ & \hspace{5mm} + \int_{0}^{\delta}\big(2\varphi(x_k) -  \varphi(x_k-z)-\varphi(x_k+z)\big) K(z) \,dz,
	\end{align*}
	where $\delta$ is a positive parameter to be chosen later. Here, we have used the odd symmetry of the kernel $K$ to write the operator in terms of the second order differences.
		
	Let us first estimate the term of the tails. If we use condition \eqref{ecuacion principio maximo 3} and the ellipticity assumption \eqref{eq: K3} we obtain
	\begin{align*}
		\int_{\delta}^{\infty} &\big(2\varphi(x_k) -  \varphi(x_k-z)-\varphi(x_k+z)\big) K(z) \,dz \leq \frac{2}{k} \int_{\delta}^{\infty} K(z)\, dz \\
		&\hspace{30mm} \leq \frac{C}{k} \left( \int_\delta^\infty \frac{1}{z^{1+2\underline{s}}}dz + \int_\delta^\infty \frac{1}{z^{1+2\overline{s}}} dz\right) \leq \frac{C}{k} \left(\delta^{-2\overline{s}}+\delta^{-2\underline{s}} \right).
	\end{align*}

	For the second integral we use the regularity of $\varphi$. Since $\varphi'$ is globally Hölder with exponent $\alpha>2\overline{s}-1\geq 2\underline{s}-1$, the second order incremental quotients satisfy
	$$|\varphi(x_k+z)+\varphi(x_k-z)-2\varphi(x_k)|\leq C |z|^{\alpha+1}.$$
	Therefore, using this estimate and the ellipticity assumption \eqref{eq: K3} we get
	\begin{align*}
		\int_{0}^{\delta}&\big(2\varphi(x_k) -  \varphi(x_k-z)-\varphi(x_k+z)\big) K(z) \,dz \leq C \int_{0}^{\delta} |z|^{\alpha+1} \, K(z) \, dz \\
		&\hspace{20mm} \leq C \left(\int_\delta^\infty \frac{z^{1+\alpha}}{z^{1+2\underline{s}}}dz + \int_\delta^\infty \frac{z^{1+\alpha}}{z^{1+2\overline{s}}} dz \right) \leq C \left(\delta^{\alpha+1-2\overline{s}}+\delta^{\alpha+1-2\underline{s}} \right).
	\end{align*}
		
	On the other hand, we use assumption \eqref{eq: V1} together with conditions $\varphi(x_k)<0$ and $\varphi(x_k) \leq \frac{1}{k}+\inf_\R\,\varphi$ to bound the zeroth order term as follows
	$$ -c(x_k)\,\varphi(x_k) \leq c_0\,\varphi(x_k) \leq \frac{c_0}{k} + c_0 \,\inf_\R\,\varphi. $$
		
	Combining all this and taking $\delta=k^{-1/2}$, we find that
	\begin{align*}
		0 &\leq L\varphi(x_k) -c(x_k)\,\varphi(x_k)  \\ &\leq C\left(k^{\overline{s}-1} + k^{(2\overline{s}-1-\alpha)/2} + k^{\underline{s}-1} + k^{(2\underline{s}-1-\alpha)/2}\right) + \frac{c_0}{k} + c_0 \,\inf_\R\,\varphi \,\,\,\,\ \text{for all }\,k\in{\Z^+}.
	\end{align*}

	Finally, by letting $k$ tend to infinity and using the assumptions $1/2\leq \underline{s}\leq \overline{s}<1$ and $\alpha>2\overline{s}-1\geq 2\underline{s}-1$ we conclude
	$$ 0 \leq c_0\,\inf_\R\varphi, $$
	which contradicts the infimum being	negative. 
\end{proof}

Odd functions are defined by their values in $(0,+\infty)$. We want to take advantage of this property to find an alternative and more useful expression for integro-differential operators when acting on such functions.

\begin{lemma} \label{Lemma: AlternativeExpressionOdd}
	Let $L$ be an integro-differential operator of the form \eqref{Tipo operador}, and let $\varphi$ be an odd function. Then,
	$$ L \varphi(x) = \int_0^\infty \big(\varphi(x)-\varphi(y)\big) \big(K(x-y)-K(x+y)\big)\,dy +  \left(2\, \int_x^\infty K(z) \, dz\right) \varphi(x).$$
\end{lemma}

Note that this alternative expression consists on a regional-type integro-differential operator in $(0,+\infty)$ plus a zeroth order term. This structure is more suitable to work with, and it will be used to establish a maximum principle in the odd setting.  As it occurs in Corollary~\ref{Lema ecuacion sigma impar}, in the odd framework $K(x-y)-K(x+y)$ takes the role of $K(x-y)$. For this reason it is natural to impose the condition $K(x-y)-K(x+y) \geq 0$ for each $x,y\in (0,+\infty)$ when working with odd functions. Actually, such a condition turns out to be equivalent to $K$ being nonincreasing in $(0,+\infty)$.

\begin{proof}[Proof of Lemma~\ref{Lemma: AlternativeExpressionOdd}]
	If we split the integral into two terms and use the odd symmetry we arrive at
	\begin{align*}
		L\varphi(x) &= \int_{-\infty}^\infty \big(\varphi(x)-\varphi(y)\big)\,K(x-y)\,dy  \\
		&= \int_{-\infty}^0 \big(\varphi(x)-\varphi(y)\big)\,K(x-y)\,dy + \int_{0}^\infty \big(\varphi(x)-\varphi(y)\big)\,K(x-y)\,dy \\
		&= \int_{0}^\infty \big(\varphi(x)-\varphi(-y)\big)\,K(x+y)\,dy + \int_{0}^\infty \big(\varphi(x)-\varphi(y)\big)\,K(x-y)\,dy \\
		&= \int_{0}^\infty \big(\varphi(x)+\varphi(y)\big)\,K(x+y)\,dy + \int_{0}^\infty \big(\varphi(x)-\varphi(y)\big)\,K(x-y)\,dy \\
		&= \int_0^\infty \big(\varphi(x)-\varphi(y)\big) \big(K(x-y)-K(x+y)\big)\,dy +  \left(2\, \int_x^\infty K(z) \, dz\right) \varphi(x).
	\end{align*}	
\end{proof}

Next, we establish an analogous maximum principle to Proposition~\ref{Lema principio maximo} in the case of odd functions. In this scenario, conditions are only imposed in the half-line since the odd symmetry transfers the information to the whole space.

\begin{proposition}{} \label{Lema principio maximo impar}
	Let $L$ be an integro-differential operator of the form \eqref{Tipo operador} with nonincreasing kernel $K$ satisfying conditions \eqref{eq: K1} and \eqref{eq: K2} for some $s \in [1/2,1)$ and $0<\lambda\leq \Lambda$. Assume the potential function $c=c(x)$ is even and satisfies \eqref{eq: V1} and 
	\begin{equation} \label{Eq: SmallDomain}
		||c||_{L^\infty(\R)} < \frac{\lambda}{s\,r_0^{2s}},
	\end{equation}
	for some positive constants $R_0>r_0>0$.
	
	For $\alpha> 2s-1$, let $\varphi$ be a bounded and $C^1$ odd function in $\R$ such that
	$[\varphi']_{C^{\alpha}(\R)} < +\infty$, 
	$$L\varphi-c\varphi \geq 0 \,\,\,\,\text{in } \,(0,r_0)\cup (R_0,+\infty),$$
	and
	$$\varphi \geq 0 \,\,\,\,\text{in } \,[r_0,R_0].$$
	
	Then,
	$$ \varphi \geq 0 \,\,\,\,\text{in } \,[0,+\infty). $$
\end{proposition}

Note that \eqref{Eq: SmallDomain} is a small domain condition, which is satisfied when $r_0$ is small enough depending on the integro-differential operator and the potential function. When applying this result in Section~\ref{sec:ProofMainResultsOdd}, such a condition will not impose any restriction since we will have enough freedom to choose $r_0>0$ as small as needed.

\begin{proof}[Proof of Proposition~\ref{Lema principio maximo impar}]
	We begin by noticing that using the previous lemma we can rewrite $L\varphi-c\varphi \geq 0$  as
	$$ \int_0^\infty \big(\varphi(x)-\varphi(y)\big) \big(K(x-y)-K(x+y)\big)\,dy - \left(c(x)-2\int_x^\infty K(z) \, dz\right)  \varphi(x) \geq 0. $$
	
	Thus, it is clear that we can repeat the proof of Proposition~\ref{Lema principio maximo} if we show that
	$$ \tilde{c}(x) := c(x)-2\int_x^\infty K(z) \, dz, $$
	satisfies
	$$\tilde{c}(x) \leq -\tilde{c}_0 < 0 \ \ \text{ in } \ \ (0,r_0)\cup (R_0,+\infty)$$
	for some positive constant $\tilde{c}_0$.
	
	On the one hand, by combining the positivity of the kernel $K$ and condition \eqref{eq: V1}, we deduce that given any $x\in(R_0,+\infty)$, 
	$$ \tilde{c}(x) \leq -c_0 < 0.$$
	On the other hand, by using the ellipticity assumption \eqref{eq: K2}, we obtain that given any $x\in(0,r_0)$,
	$$ \tilde{c}(x) \leq ||c||_{L^\infty(\R)} - 2\lambda \int_x^\infty z^{-1-2s}\,dz = ||c||_{L^\infty(\R)} - \frac{\lambda}{s}x^{-2s}\leq ||c||_{L^\infty(\R)} - \frac{\lambda}{s}r_0^{-2s}<0. $$
	
	Hence, it is enough to take $\tilde{c}_0 = \min\left\{c_0,\frac{\lambda}{s}r_0^{-2s}-||c||_{L^\infty(\R)}\right\}>0$.
\end{proof}

Let us remark that a maximum principle as in Proposition~\ref{Lema principio maximo impar} cannot hold if we remove the odd symmetry of the function. In that case, having a negative minimum in $(0,+\infty)$ does not give any information about the sign of the operator at this point since the behavior of the function in $(-\infty,0)$ is unknown.


\section{Integrability bounds for the kernel} \label{sec:Integrability}
This section is devoted to presenting some integrability bounds that will be needed to establish Theorems \ref{Teorema L0} and \ref{Mi teorema general impar}. In fact, the validity of these bounds is what prevents us from extending our results to $s\in (0,1/2)$.

In \cite{HamelRosOtonSireValdinoci}, Hamel, Ros-Oton, Sire, and Valdinoci work with compactly supported kernels in dimension $2$. Once such a condition is assumed, the integrability bounds for the kernel follow immediately for free. In our case, when removing that assumption, some estimates become much more delicate. In order to control the integrals we define some auxiliary sets and prove certain relations between them that simplify the computations.

First, we show the following identity:

\begin{lemma} \label{Lema conjuntos 1}
	Let $S_R$, $D_R$, $\mathcal{T}_R^x$, and $\mathcal{T}_R^y$  be the sets
	$$ S_R =\left(B_{2R}\times B_R^c\right) \cup \left(B_R^c\times B_{2R}\right) \subset \R^n\times \R^n, $$
	$$ D_R = \left\{(x,y)\in\R^n\times\R^n : |x-y|\leq 4R\right\}\subset \R^n\times \R^n, $$
	$$ \mathcal{T}_R^x = \left\{(x,y)\in\R^n\times\R^n\, \text{ s.t. } \,|x|<2R \,\text{ and }\, |x-y|\geq 4R\right\}\subset \R^n\times \R^n, $$
	and
	$$ \mathcal{T}_R^y = \left\{(x,y)\in\R^n\times\R^n\, \text{ s.t. } \,|y|<2R \,\text{ and }\, |x-y|\geq 4R\right\}\subset \R^n\times \R^n. $$
	
	Then, $\mathcal{T}_R^x$ and $\mathcal{T}_R^y$ are disjoint and satisfy
	$$ S_R \setminus D_R = \mathcal{T}_R^x \cup \mathcal{T}_R^y. $$
\end{lemma}

\begin{proof}
	On the one hand, let $(x,y)\in S_R\setminus D_R$. By the symmetry of the set with respect to $x$ and $y$ we can assume without loss of generality that $(x,y)\in\left(B_{2R}\times B_R^c\right)\cap\{|x-y|>4R\}$. Then, $(x,y)\in \mathcal{T}_R^x$ follows trivially. 
	
	On the other hand, given $(x,y)\in \mathcal{T}_R^x$, we can apply the triangle inequality to deduce that $|y|\geq 2R$. Therefore, we conclude that $(x,y)\in S_R\setminus D_R$.
	
	Finally, in order to prove that the sets $\mathcal{T}_R^x$ and $\mathcal{T}_R^y$ are disjoint we only need to recall that given $(x,y)\in \mathcal{T}_R^x$, it satisfies $|y|\geq 2R$, and therefore $(x,y)\not\in \mathcal{T}_R^y$.
\end{proof}

Next, we prove a useful inclusion of sets.

\begin{lemma} \label{Lema conjuntos 2}
	Let $S_R$ and $D_R$ be as in Lemma \ref{Lema conjuntos 1}, and let $\mathcal{R}_R^x$ and $\mathcal{R}_R^y$ be the sets
	$$ \mathcal{R}_R^x = \left\{(x,y)\in\R^n\times\R^n\, \text{ s.t. } \,|x|<R \,\text{ and }\, |x-y|\leq 2R\right\}\subset \R^n\times \R^n, $$
	and
	$$ \mathcal{R}_R^y = \left\{(x,y)\in\R^n\times\R^n\, \text{ s.t. } \,|y|<R \,\text{ and }\, |x-y|\leq 2R\right\}\subset \R^n\times \R^n. $$
	
	Then,
	$$ \mathcal{R}_{2R}^x \setminus \mathcal{R}_{R}^x \subseteq S_R \cap D_R \subseteq \left(\mathcal{R}_{2R}^x \setminus \mathcal{R}_{2R/3}^x\right) \cup \left(\mathcal{R}_{2R}^y \setminus \mathcal{R}_{2R/3}^y\right)  $$
\end{lemma}

\begin{proof}
	The proof of these inclusions is simple. As in Lemma~\ref{Lema conjuntos 1}, we only need to consider different cases and use the triangle inequality to relate $|x|$, $|y|$, and $|x-y|$ .
	
	For the first inclusion, let $(x,y)\in \mathcal{R}_{2R}^x \setminus \mathcal{R}_{R}^x$. We distinguish two cases: either $|x|\leq R$ and $2R\leq |x-y|\leq 4R$, or $R\leq |x|\leq 2R$ and $|x-y|\leq 4R$. In the first scenario, it is clear by using the triangle inequality that $|y|\geq R$, and therefore $(x,y) \in (B_{2R}\times B_R^c)\cap D_R \subset S_R\cap D_R$. In the second one, we only need to note that $(B_{2R}\setminus B_R)\times \R^n \subset (B_{2R}\times B_R^c) \cup (B_{R}^c\times B_{2R})$.
	
	For the second inclusion, by taking advantage of the symmetry with respect to $x$ and $y$ of the sets $S_R$ and $D_R$ it is enough to prove that $\left(B_{2R}\times B_R^c\right) \cap D_R \subset (\mathcal{R}_{2R}^x \setminus \mathcal{R}_{2R/3}^x) \cup (\mathcal{R}_{2R}^y \setminus \mathcal{R}_{2R/3}^y)$. Then, given $(x,y) \in \left(B_{2R}\times B_R^c\right) \cap D_R$, if $4/3R\leq|x-y|\leq 4R$ or $2R/3\leq |x|\leq 2R$, it is clear that $(x,y) \in \mathcal{R}_{2R}^x \setminus \mathcal{R}_{2R/3}^x $. Therefore, we are left with proving the desired result for the case $|x|\leq 2R/3$, $|y|\geq R$, and $|x-y|\leq 4/3R$. By applying the triangle inequality we can deduce that in such a case $|y|\leq 2R$ and we conclude that $(x,y)\in \mathcal{R}_{2R}^y \setminus \mathcal{R}_{2R/3}^y$.		
\end{proof}

Once we have established the previous relations of sets, we can proceed by proving the integral estimates. We first state them for the kernel of the fractional Laplacian. The case of general integro-differential operators will follow from them as a consequence of the ellipticity assumptions.

\begin{lemma}{} \label{Lema integrabilidad}
	Let $S_R$ and $D_R$ be as in Lemma~\ref{Lema conjuntos 1} and Lemma~\ref{Lema conjuntos 2}. Assume $s\in(0,1)$ and $0\leq \gamma\leq \min(s,1/2)$. 
	
	Then,
	$$ \int_{S_R\cap D_R} \frac{|x|^{2\gamma}}{|x-y|^{n+2s-2}}\,dx dy \leq C \,R^{2\gamma+n+2-2s}, $$
	and
	$$ \int_{S_R \setminus D_R}\frac{|x|^{2\gamma}}{|x-y|^{n+2s}}\,dx dy \leq C \,R^{2\gamma+n-2s}, $$
	where $C$ is a positive constant depending only on $n$, $s$, and $\gamma$.
\end{lemma}

We point out that analogous bounds from below can also be deduced. However, since we will not use such estimates in the present work, we skip them.

\begin{proof}[Proof of Lemma~\ref{Lema integrabilidad}]
	To obtain the first estimate we use the inclusion of sets given by Lemma~\ref{Lema conjuntos 2}. That is,
	\begin{align*}
		\int_{S_R\cap D_R} \frac{|x|^{2\gamma}}{|x-y|^{n+2s-2}}&\,dx dy \leq C R^{2\gamma} \int_{S_R\cap D_R} |x-y|^{2-n-2s}\,dx dy \\
		&\hspace{-10mm}\leq C R^{2\gamma}  \left( \int_{ \mathcal{R}_{2R}^x \setminus \mathcal{R}_{2R/3}^x} |x-y|^{2-n-2s}\,dx dy + \int_{ \mathcal{R}_{2R}^y \setminus \mathcal{R}_{2R/3}^y}  |x-y|^{2-n-2s}\,dx dy\right) \\
		&\hspace{-10mm}\leq  C R^{2\gamma}\int_{ \mathcal{R}_{2R}^x \setminus \mathcal{R}_{2R/3}^x} |x-y|^{2-n-2s}\,dx dy \\
		&\hspace{-10mm} =   C R^{2\gamma} \left(\int_{ \mathcal{R}_{2R}^x } |x-y|^{2-n-2s}\,dx dy - \int_{ \mathcal{R}_{2R/3}^x} |x-y|^{2-n-2s}\,dx dy\right) \\
		&\hspace{-10mm}= C R^{2\gamma} \left( (2R)^{n-2s+2}-(2R/3)^{n-2s+2}\right) \\
		&\hspace{-10mm}= C \,R^{2\gamma+n+2-2s}.
	\end{align*}
	
	The second bound is more delicate. First we find that
	\begin{align*}
		\int_{\mathcal{T}_{2R}^x}\frac{|x|^{2\gamma}}{|x-y|^{n+2s}}\,dx dy &\leq C R^{2\gamma} \int_{\mathcal{T}_{2R}^x } |x-y|^{-n-2s}\,dx dy, \\
		&= C R^{2\gamma} \int_{B_{2R}} dw \int_{B_{4R}^c} |z|^{-n-2s} dz \\
		&= C\,C R^{2\gamma+n} \int_{4R}^\infty r^{-n-2s}r^{n-1} \,dr \\
		&= C\,R^{2\gamma+n-2s},
	\end{align*}
	where we have performed the change of variables: $z=x-y$ and $w=x$. Next, we obtain
	\begin{align*}
		\int_{\mathcal{T}_{2R}^y}\frac{|x|^{2\gamma}}{|x-y|^{n+2s}}\,dx dy &=  \int_{B_{2R}} dw \int_{B_{4R}^c} dz \frac{|w+z|^{2\gamma}}{|z|^{n+2s}} \\
		&\leq \int_{B_{2R}} dw \int_{B_{4R}^c} dz \frac{|w|^{2\gamma}+|z|^{2\gamma}}{|z|^{n+2s}} \\
		&\leq  C\, R^{n} \left( R^{2\gamma}\int_{4R}^\infty r^{-n-2s}r^{n-1} \,dr + \int_{4R}^\infty r^{-n-2s+2\gamma}r^{n-1} \,dr\right)\\
		&= C\,R^{2\gamma+n-2s}.
	\end{align*}

	Finally, we conclude the proof by applying Lemma~\ref{Lema conjuntos 1}. Let us point out that it is crucial in the last estimate to assume $\gamma\leq \min(s,1/2)$ in order to ensure the integrability.
\end{proof}

Once we have established the previous bounds for the kernel of the fractional Laplacian, we can easily obtain the estimates we need, in cross-shaped domains, for the bigger class of operators satisfying condition \eqref{eq: K3}.

\begin{corollary}{} \label{Corolario integrabilidad}
	Let $L$ be an integral operator of the form \eqref{Tipo operador}, with kernel $K$ satisfying conditions \eqref{eq: K1} and \eqref{eq: K3} for some $0<\underline{s}\leq \overline{s}<1$. Assume the set $S_R$ is defined as in Lemma~\ref{Lema conjuntos 1} and $0\leq \gamma\leq \min(\underline{s},1/2)$. 
	
	Then
	$$\int_{S_R} \min\left\{1, \frac{|x-y|}{R}  \right\}^2 |x|^{2\gamma}\, K(x-y) \,dx dy \leq C\,R^{n+2\gamma-2\underline{s}}, $$
	for a positive constant $C$ not depending on $R$. 
	
	In particular, if $n=1$, $1/2\leq\underline{s}\leq \overline{s}<1$, and $\gamma\in [0,\underline{s}-1/2]$, there is a positive constant $C$, independent of $R$, such that
	$$\int_{S_R} \min\left\{1, \frac{|x-y|}{R}  \right\}^2 |x|^{2\gamma}\,K(x-y) \,dx dy \leq C, $$
	for any $R\geq 1$.
\end{corollary}

Note that the uniform bound can only be established when $n+2\gamma-2\underline{s}\leq 0$. Since the dimension $n$ is an integer, it means that the previous condition is not satisfied unless $n=1$, $1/2\leq\underline{s}\leq \overline{s}<1$, and $\gamma\in [0,\underline{s}-1/2]$. This is the reason why we need to assume such dimension and range of fractional powers, in addition to a growth condition of order $s-1/2$ in Theorem~\ref{Teorema L0}.

\begin{proof}[Proof of Corollary \ref{Corolario integrabilidad}]
	First, note that
	$$ \min\left\{1, \frac{|x-y|}{R}  \right\} = \begin{cases}
	\frac{|x-y|}{R} &\text{if $(x,y) \in D_R$,}\\
	1 &\quad\text{otherwise,}
	\end{cases}$$
	where $D_R$ is the set defined in Lemma \ref{Lema integrabilidad}.
	
	Then, by the linearity of the integral, the ellipticity assumption in the kernel \eqref{eq: K3}, and the relations of sets from Lemma \ref{Lema conjuntos 1} we get
	
	\begin{align*}
		&\int_{S_R} \min\left\{1, \frac{|x-y|}{R}  \right\}^2|x|^{2\gamma}\,K(x-y) \,dx dy \\
		&\hspace{12mm}= \int_{S_R \cap D_R} \frac{|x-y|^2}{R^2}|x|^{2\gamma}\,K(x-y) \,dx dy + \int_{S_R\setminus D_R} |x|^{2\gamma}\, K(x-y)\,dx dy \\
		&\hspace{12mm}\leq \Lambda_1 \left(  \int_{S_R \cap D_R} \frac{|x|^{2\gamma}}{R^2|x-y|^{n+2\underline{s}-2}} \,dx dy + \int_{S_R\setminus D_R} \frac{|x|^{2\gamma}}{|x-y|^{n+2\underline{s}}}\,dx dy \right) \\
		&\hspace{15mm} + \Lambda_2 \left(  \int_{S_R \cap D_R} \frac{|x|^{2\gamma}}{R^2|x-y|^{n+2\overline{s}-2}} \,dx dy + \int_{S_R\setminus D_R} \frac{|x|^{2\gamma}}{|x-y|^{n+2\overline{s}}}\,dx dy \right) \\
		&\hspace{12mm}\leq \Lambda_1\,C_{n,\underline{s}} \,R^{n+2\gamma-2\underline{s}} + \Lambda_2\,C_{n,\overline{s}} \,R^{n+2\gamma-2\overline{s}} \leq C\,R^{n+2\gamma-2\underline{s}}.
	\end{align*}  
\end{proof}

Finally, we establish an analogue result in the odd setting.

\begin{corollary}{} \label{Corolario integrabilidad impar}
	Let $L$ be an integral operator of the form \eqref{Tipo operador}, with kernel $K$ being radially decreasing and satisfying conditions \eqref{eq: K1} and \eqref{eq: K2}. Assume $n=1$, $0\leq \gamma\leq \min(s,1/2)$, and the set
	$$S_R^{++} = S_R\cap \left(\R^+\times \R^+\right)$$
	with $S_R$ as in the previous results.
	
	Then,
	$$\int_{S_R^{++}} \min\left\{1, \frac{|x-y|}{R}  \right\}^2|x|^{2\gamma}\,\big( K(x-y)-K(x+y) \big) \,dx dy \leq C\,R^{1+2\gamma-2s}, $$
	for a positive constant $C$ not depending on $R$. In particular, in the case $s\in[1/2,1)$ and $0\leq\gamma\leq s-1/2$
	$$\int_{S_R^{++}} \min\left\{1, \frac{|x-y|}{R}  \right\}^2|x|^{2\gamma}\,\big( K(x-y)-K(x+y) \big) \,dx dy \leq C, $$
	for any $R\geq 1$.
\end{corollary}

\begin{proof}
	By using Lemma \ref{Lema integrabilidad}, the ellipticity condition of the kernel and the
	symmetries of the domain $S_R$ with respect to $x$ and $y$ we get
	\begin{align*}
		\int_{S_R^{++}} \min\left\{1, \frac{|x-y|}{R}  \right\}^2&|x|^{2\gamma}\,\big( K(x-y)-K(x+y) \big)  \,dx dy \\
		&\hspace{-3mm} \leq \int_{S_R^{++}} \min\left\{1, \frac{|x-y|}{R}  \right\}^2|x|^{2\gamma}\,\big( K(x-y)+K(x+y) \big)  \,dx dy \\
		&\hspace{-3mm }= \int_{S_R^{++}} \min\left\{1, \frac{\big||x|-|y|\big|}{R}  \right\}^2|x|^{2\gamma}\,\left( K(x-y)+K(x+y) \right) \,dx dy \\
		&\hspace{-3mm} = \frac{1}{2} \int_{S_R} \min\left\{1, \frac{\big||x|-|y|\big|}{R}  \right\}^2|x|^{2\gamma}\,K(x-y) \,dx dy \\
		&\hspace{-3mm} \leq \frac{1}{2} \int_{S_R} \min\left\{1, \frac{|x-y|}{R}  \right\}^2|x|^{2\gamma}\,K(x-y) \,dx dy \\
		&\hspace{-3mm} \leq \frac{\Lambda}{2} \int_{S_R} \min\left\{1, \frac{|x-y|}{R}  \right\}^2|x|^{2\gamma}\,K(x-y) \,dx dy \\
		&\hspace{-3mm} \leq \Lambda \,C_s \,R^{1+2\gamma-2s}.
	\end{align*}
\end{proof}


\section{Proof of Theorems~\ref{Teorema L0} and \ref{Mi teorema general impar}} \label{sec:ProofMainResults}

This section is devoted to proving the results presented in Section~\ref{Sec:Introduction} where no symmetries are assumed.

In order to deal with the first scenario in Theorem~\ref{Teorema L0} we first show that the quotient of two bounded solutions is also bounded:

\begin{proposition}{} \label{Proposicion acotacion sigma}
	Let $L$ be an integro-differential operator of the form \eqref{Tipo operador} satisfying the symmetry and ellipticity conditions \eqref{eq: K1} and \eqref{eq: K3} for some $1/2\leq \underline{s} \leq \overline{s}<1$. Assume that the potential function $c=c(x)$ satisfies condition \eqref{eq: V1} for some positive constant $R_0$. 
	
	For $\alpha>2\overline{s}-1$, let $w$ and $\widetilde{w}$ be two bounded and $C^1$ functions such that $[w']_{C^{\alpha}(\R)}$ and $[\widetilde{w}']_{C^{\alpha}(\R)}$ are finite. In addition, assume that
	$$ w > 0 \,\,\,\,\text{in } \,[-R_0,R_0],$$
	$$Lw-cw \geq 0 \,\,\,\,\text{in } \,\R\setminus[-R_0,R_0],$$
	and
	$$L\widetilde{w} -c\widetilde{w} = 0 \,\,\,\,\text{in } \,\R\setminus[-R_0,R_0].$$
	
	Then, there exists a positive constant $C$ such that
	$$ \left| \frac{\widetilde{w}}{w} \right| \leq C \,\,\,\,\text{in } \,\R. $$
\end{proposition}

\begin{proof}
	First, by applying Proposition~\ref{Lema principio maximo} and the strong maximum principle we deduce
	$$ w > 0 \,\,\,\,\text{in } \,\R. $$
	As a consequence, the quotient $\tilde{w}/w$ is well-defined and continuous in the whole real line. 
	
	Next, we prove that such a quotient is indeed bounded. This will follow after showing the positivity of the functions
	$$ \varphi_\pm = C\,w \pm \widetilde{w}, $$
	where $C$ is a nonnegative constant to be chosen. Note that these functions inherit the regularity of $w$ and $\widetilde{w}$ from being a linear combination of them.
	
	Let us take $C\geq 0$ satisfying
	$$ C \geq \left|\left| \frac{\widetilde{w}}{w} \right|\right|_{L^\infty(-R_0,R_0)}. $$
	It is clear by definition that $ \varphi_\pm\geq 0$ in $[-R_0,R_0]$. Moreover,
	$$ L \varphi_\pm -c\varphi_\pm = C \big(Lw-cw\big) \pm \big(L\widetilde{w}-c\widetilde{w}\big) =  C \big(Lw-cw\big) \geq 0 \,\,\,\,\,\text{in } \,\R\setminus[-R_0,R_0].$$
	Hence, by applying Proposition~\ref{Lema principio maximo} to $\varphi_\pm$ we conclude that
	$$ \varphi_\pm = C\,w \pm \widetilde{w} \geq 0 \,\,\,\,\,\,\, \text{ in } \,\R, $$
	which is equivalent to
	$$ \left| \frac{\widetilde{w}}{w} \right| \leq C \,\,\,\,\text{in } \,\R. $$
\end{proof}

Next, we establish the uniqueness result for the linear equation \eqref{Ecuacion Lineal}. As already explained in the introduction, we present here a more general result from which we will deduce Theorem~\ref{Teorema L0} among others. On the one hand, the ellipticity condition on the kernel is relaxed to \eqref{eq: K3}, which means the kernel being bounded only from above, even with different order at the origin and infinity. On the other hand, it is not needed the existence of a positive solution but a positive supersolution.

\begin{theorem}{} \label{Mi teorema general}
	Let $L$ be an integro-differential operator of the form \eqref{Tipo operador} satisfying the symmetry and ellipticity conditions \eqref{eq: K1} and \eqref{eq: K3} for some $1/2\leq \underline{s}\leq \overline{s}<1$.  
	
	For $\alpha>2\overline{s}-1$, let $w$ and $\widetilde{w}$ be $C^{1,\alpha}$ functions in $\R$. Assume that 
		\begin{itemize}[leftmargin=*]
			\item[$\bullet$]  either $w$ and $\widetilde{w}$ are both bounded and such that $[w']_{C^{\alpha}(\R)}$ and $[\widetilde{w}']_{C^{\alpha}(\R)}$ are finite, $w>0$, and the potential function $c=c(x)$ satisfies condition \eqref{eq: V1};
			\item[$\bullet$]  or $w$ is such that
			$$ 0 < C^{-1} \leq w(x) \leq C\,\,\,\,\,\,\, \text{in} \,\,\,\R,    $$
			and $\tilde{w}$ satisfies the growth condition
			$$ ||\widetilde{w}||_{L^\infty(-R,R)} \leq C R^{\underline{s}-1/2}, \, \, \text{ for every } \, \, R>1$$
			for some positive constant $C$.
		\end{itemize}
	
	In addition, assume that
	$$Lw-cw \geq 0 \,\,\,\,\text{in } \,\R,$$
	and
	$$L\widetilde{w}-c\widetilde{w} = 0 \,\,\,\,\text{in } \,\R.$$

	Then
	$$ \frac{\widetilde{w}}{w} \equiv \text{constant}. $$
\end{theorem}

In the following proof, once the boundedness of $\sigma := \widetilde{w}/w$ (Proposition~\ref{Proposicion acotacion sigma}) and some integrability estimates (Lemma \ref{Lema integrabilidad}) are established, it will be enough to follow the strategy developed in \cite{HamelRosOtonSireValdinoci} to conclude that such a quotient is constant.

\begin{proof}[Proof of Theorem \ref{Mi teorema general}]

We begin by noticing that, using the bounds on $w$ and $\widetilde{w}$, and applying Proposition~\ref{Proposicion acotacion sigma}, we immediately deduce that $\sigma = \widetilde{w}/w$ satisfies the growth condition $|\sigma(x)|\leq C |x|^{\underline{s}-1/2}$. This is the first step to show that $\sigma$ is constant.

	Let $\eta$ be a $C^\infty$ function on $[0,+\infty)$ such that
	$0\le\eta\le 1$ and
	$$\eta = \begin{cases}
	1 &\text{if $0\le t\le 1$,}\\
	0 &\quad\text{if $t\ge 2$.}
	\end{cases}$$
	For each $R>1$, we take $\eta_R(x)=\eta \left( \frac{\vert x\vert}{R}\right)$. It is clear that it satisfies the pointwise estimate
\begin{align} \label{ecuacion minimo}
	|\eta_R(x)-\eta_R(y)| \leq C \min\left\{1, \frac{|x-y|}{R}\right\} \, \, \text{ for every } \, \, x,y\in \R
\end{align}
and some positive constant $C$ depending only on $\eta$.

	Next, we apply Lemma~\ref{Lema ecuacion sigma} with $\tau = \eta_R$ to deduce
	\begin{align*}
	0 \leq J_1 :&=  \int_\R\int_\R \,\big( \sigma(x)-\sigma(y) \big)^2\big(\eta_R^2(x)+\eta_R^2(y) \big)\,w(x)\,w(y)\,K(x-y)\,\,dx dy \\ &\leq -\int_\R\int_\R \big( \sigma^2(x)-\sigma^2(y) \big)\big( \eta_R^2(x)-\eta_R^2(y)\big)w(x)\,w(y)\,K(x-y)\,dx dy \\ &\leq
	\int_\R \int_\R |\sigma(x)-\sigma(y)||\sigma(x)+\sigma(y)||\eta_R(x)-\eta_R(y)||\eta_R(x)+\eta_R(y)|\cdot\\ &\hspace{73mm} \cdot w(x)\,w(y)\,K(x-y)\,dx dy \\ &=
	\int_{S_R} |\sigma(x)-\sigma(y)||\sigma(x)+\sigma(y)||\eta_R(x)-\eta_R(y)||\eta_R(x)+\eta_R(y)|\cdot\\ &\hspace{73mm} \cdot w(x)\,w(y)\,K(x-y)\,dx dy \\ &=: J_2.
	\end{align*}
	Note that the last equality follows from the support of $|\eta_R(x)-\eta_R(y)|$ being the set $S_R$ defined in Lemma \ref{Lema conjuntos 1}.
	
	Furthermore, by using Cauchy-Schwartz inequality we get
	\begin{align*}
	J_2^2 & \leq \int_{S_R} \big( \sigma(x)-\sigma(y) \big)^2\big(\eta_R(x)+\eta_R(y) \big)^2\,w(x)\,w(y)\,K(x-y)\,dx dy \,\cdot \\
	&\hspace{11mm} \cdot \int_{S_R} \big( \sigma(x)+\sigma(y) \big)^2\big(\eta_R(x)-\eta_R(y) \big)^2\,w(x)\,w(y)\,K(x-y)\,dx dy \\&\leq
	2\, J_1 \int_{S_R} \big(\sigma(x)+\sigma(y)\big)^2\big(\eta_R(x)-\eta_R(y) \big)^2\,w(x)\,w(y)\,K(x-y)\,dx dy.
	\end{align*}
	
	Now, by combining the boundedness of $w$, the growth condition on $\sigma$, the pointwise estimate \eqref{ecuacion minimo} for $\eta_R$, and the integrability result from Corollary~\ref{Corolario integrabilidad}, we find
	\begin{align*}
	\int_{S_R} \big( \sigma(x)+\sigma(y) \big)^2&\big( \eta_R(x)-\eta_R(y) \big)^2\,w(x)\,w(y)\,K(x-y)\,dx\,dy \leq \\ 
	&\hspace{-4mm}\leq C \,\int_{S_R}\big( \eta_R(x)-\eta_R(y) \big)^2\,|\sigma(x)|^2\,K(x-y)\,dx\,dy \leq C.
	\end{align*}
	
	Summarizing, we have
	$$ 0 \leq J_1^2 \leq J_2^2 \leq C\,J_1, $$
	which leads to
	$$ J_1 =  \int_\R\int_\R \big( \sigma(x)-\sigma(y) \big)^2\big(\eta_R^2(x)+\eta_R^2(y) \big)\,w(x)\,w(y)\,K(x-y)\,dx\,dy \leq C. $$
	In particular, since $\eta_R=1$ in $B_R$, we deduce
	$$ \int_{B_R}\int_{B_R} \big( \sigma(x)-\sigma(y) \big)^2\,w(x)\,w(y)\,K(x-y)\,dx\,dy \leq C, $$
	where $C$ is a positive constant not depending on $R$. From that estimate and the monotone convergence theorem we obtain that $\big( \sigma(x)-\sigma(y) \big)^2\,w(x)\,w(y)\,K(x-y)$ belongs to $L^1(\R\times \R)$. Hence, we conclude from the dominated convergence theorem that
	$$ \lim_{R\to\infty} \int_{S_R} \big( \sigma(x)-\sigma(y) \big)^2\,w(x)\,w(y)\,K(x-y) dx\,dy = 0. $$
	
	Combining all this together, we arrive at
	\begin{align*}
	&\left[ \int_\R\int_\R \big( \sigma(x)-\sigma(y) \big)^2\,w(x)\,w(y)\,K(x-y) dx dy \right] ^2 = \\ &\hspace{11mm}= \frac{1}{2} \lim_{R\to\infty} \left[  \int_\R\int_\R \big( \sigma(x)-\sigma(y) \big)^2\big( \eta_R^2(x)+\eta_R^2(y) \big)\,w(x)\,w(y)\,K(x-y) dx dy \right] ^2\\
	&\hspace{11mm}\leq C\,\lim_{R\to\infty}  \int_{S_R} \big( \sigma(x)-\sigma(y) \big)^2\big( \eta_R^2(x)+\eta_R^2(y) \big)\,w(x)\,w(y)\,K(x-y) dx dy \\
	&\hspace{11mm} \leq C\,\lim_{R\to\infty}  \int_{S_R} \big( \sigma(x)-\sigma(y) \big)^2w(x)\,w(y)\,K(x-y) dx dy = 0.
	\end{align*}
	From this and the positivity of both $w$ and $K$, we obtain that $\big(\sigma(x)-\sigma(y)\big)^2=0$ for almost every $(x,y) \in \R\times \R$. Thus, by continuity, we conclude that $$ \sigma = \frac{\tilde{w}}{w} \equiv \text{constant}. $$
\end{proof}

Using the previous result we can easily deduce Theorem~\ref{Teorema L0}. In fact, we only need to check that solutions from Theorem~\ref{Teorema L0} have the required regularity to apply Theorem~\ref{Mi teorema general}. Such property will follow thanks to the regularizing effect of the operators satisfying the ellipticity assumption \eqref{eq: K2}.

\begin{proof}[Proof of Theorem~\ref{Teorema L0}]
	In order to prove Theorem~\ref{Teorema L0} we will show that any bounded solution to the linear equation \eqref{Ecuacion Lineal} with $L$ being of the form \eqref{Tipo operador} and satisfying \eqref{eq: K1} and \eqref{eq: K2} is globally Hölder continuous with exponent $\alpha+1 > 2s$ (we use here the notation $C^\gamma = C^{\lfloor \gamma \rfloor,\gamma-\lfloor \gamma \rfloor}$ whenever $\gamma>1$). From this and Theorem~\ref{Mi teorema general}, the uniqueness result will follow.
	
	The proof of the regularity is based on defining the auxiliary function $f(x):=c(x)u(x)$ and using the interior regularity results from \cite{Serra-C2s+alphaRegularity} for the nonlocal equation
	$$ Lu = f \,\,\,\,\, \text{in} \,\, B_1\subset\R^n. $$
	
	Let us first prove that any solution $u$ satisfies $||u||_{C^{\beta}(\R)}<\infty$ for each $\beta<2s$. The boundedness of both $u$ and $c$ leads to $f\in L^\infty(\R)$. Thus, we can apply Corollary 1.2 from \cite{Serra-C2s+alphaRegularity} for each unitary ball in $\R$ to conclude that
	\begin{align*}
		||u||_{C^\beta\left(B_{1/2}(x_0)\right)} &\leq C \left( ||f||_{L^\infty\left(B_{1}(x_0)\right)} + ||u||_{L^\infty\left(\R\right)} \right) \\&\leq C \left( ||f||_{L^\infty\left(\R\right)} + ||u||_{L^\infty\left(\R\right)} \right)
	\end{align*}
	for any given point $x_0 \in \R$ and $\beta<2s$.
	
	In particular, we know that $||u||_{C^{\beta_0}(\R)}$ is finite. Hence, we can use the fact that $||c||_{C^{\beta_0}(\R)}$ is also finite to deduce that $f$ inherits such a property and apply Theorem~1.1 from \cite{Serra-C2s+alphaRegularity} to establish
	\begin{align*}
		||u||_{C^{2s+\beta_0}\left(B_{1/2}(x_0)\right)} &\leq C \left( ||f||_{C^{\beta_0}\left(B_{1}(x_0)\right)} + ||u||_{C^{\beta_0}\left(\R\right)} \right) \\&\leq C \left( ||f||_{C^{\beta_0}\left(\R\right)} + ||u||_{C^{\beta_0}\left(\R\right)} \right).
	\end{align*}

	Finally, if we take $\alpha:=2s+\beta_0-1$, we can apply Theorem~\ref{Mi teorema general} to deduce the uniqueness of solution, concluding the proof.
\end{proof}

We state now an interesting consequence of Theorem~\ref{Mi teorema general} which is not included in Theorem~\ref{Teorema L0}. It deals with sums of fractional Laplacians.

\begin{corollary}{} \label{Corolario suma fraccionarios}
	Let $L$ be a nonlocal operator of the form
	$$ Lu = \int_{\underline{s}}^{\overline{s}} (-\Delta)^su\,d\mu(s), $$
	with $1/2\leq \underline{s} \leq \overline{s} < 1$, where $\mu$ is a probability measure supported in
	$[\underline{s},\overline{s}]$, i.e.,
	$$ \mu\geq 0 \,\,\, \text{ and } \,\,\, \mu\left([\underline{s},\overline{s}]\right) = \mu(\R) =1.$$
	Assume that $c$ is bounded in $\R$, satisfies condition \eqref{eq: V1}, and $||c||_{C^{1,2\overline{s}-1}(\R)} < +\infty$.
	
	Let $w$ and $\widetilde{w}$ be two bounded solutions of the linear equation
	$$ L\varphi-c(x)\varphi=0 \,\,\,\,\,\,\, \text{in} \,\,\,\R, $$
	with $w>0$.
	Then
	$$ \frac{\widetilde{w}}{w} \equiv \text{constant}. $$
\end{corollary}

\begin{proof}
	In order to establish Corollary~\ref{Corolario suma fraccionarios} we only need to show that the operator $L$ is of the form \eqref{Tipo operador} satisfying \eqref{eq: K1} and \eqref{eq: K3} and that bounded solutions of the linear equation are globally H\"older continuous with exponent grater than $2\overline{s}$.
	
	First, let us rewrite the expression of $L$ in an alternative way:
	\begin{align*}
		Lu &:= \int_{\underline{s}}^{\overline{s}} (-\Delta)^su\,d\mu(s) = \int_{\underline{s}}^{\overline{s}} \left[ \int_\R \frac{u(x)-u(y)}{|x-y|^{1+2s}} dy \right]d\mu(s)\\
		& = \int_\R \big(u(x)-u(y)\big) \left( \int_{\underline{s}}^{\overline{s}} \frac{d\mu(s)}{|x-y|^{1+2s}}   \right) dy.
	\end{align*}
	Thus, $L$ is an integral operator of the form \eqref{Tipo operador} with kernel
	$$ K(z) = \int_{\underline{s}}^{\overline{s}} \frac{d\mu(s)}{|z|^{1+2s}}. $$

	Moreover, it satisfies conditions \eqref{eq: K1} and \eqref{eq: K3}. Indeed,
	\begin{align*}
		K(z) &\leq \int_{\underline{s}}^{\overline{s}} \frac{d\mu(s)}{|z|^{1+2\overline{s}}} \chi_{\{|z|\leq 1\}}(z) + \int_{\underline{s}}^{\overline{s}} \frac{d\mu(s)}{|z|^{1+2\underline{s}}} \chi_{\{|z|\geq 1\}}(z) \\
		&= \frac{1}{|z|^{1+2\overline{s}}} \chi_{\{|z|\leq 1\}}(z) + \frac{1}{|z|^{1+2\underline{s}}} \chi_{\{|z|\geq 1\}}(z) \\
		&\leq \frac{1}{|z|^{1+2\overline{s}}} + \frac{1}{|z|^{1+2\underline{s}}}.
	\end{align*}

Next, let us apply the regularity results from \cite{CabreSerra-SumLaplacians} to deduce the Hölder regularity of bounded solutions. Since $c$ is a $C^{1,2\overline{s}-1}$ function we can use a standard bootstrap argument that leads to the desired regularity of the solution after using Lemma~2.1 from \cite{CabreSerra-SumLaplacians} $\left \lfloor{2\overline{s}/\alpha}\right \rfloor+1$ times, where $\alpha$ is a positive constant depending only $\underline{s}$ and $n$. Thus, we conclude that $u$ belongs to $C^{1,2\overline{s}-1+\beta}$ in $\R$ with $\beta = \left \lfloor{2\overline{s}/\alpha}\right \rfloor\alpha+\alpha-2\overline{s}>0$.
	
Combining all this, we can apply Theorem \ref{Mi teorema general} to establish the uniqueness of solutions.
\end{proof}

Finally, we prove the nondegeneracy of layer solutions, Theorem~\ref{Nondegeneracy Layer}.

\begin{proof}[Proof of Theorem~\ref{Nondegeneracy Layer}]
	First, we know by Theorem 1 in
	\cite{CozziPassalacqua} that the layer solution $u$ is a $C^{2,2s-1+\gamma}$ function for some $\gamma>0$ and $u'$ is bounded in the whole line.
	
	We need to show that $u'$ is the unique bounded solution to
	\begin{equation} \label{linearized equation 1}
		L v - f'(u)v = 0 \,\,\,\,\,\, \text{in} \,\,\R.
	\end{equation}
	
	Let us take $c(x) = f'(u(x))$. We only need to check that the hypotheses in Theorem~\ref{Teorema L0} are satisfied. Since $f'\in C^\gamma([-1,1])$ and $u$ is a continuous and bounded function, it is clear that $c$ is bounded and such that $[c]_{C^\gamma(\R)}$ is finite. Furthermore
	$$ \lim_{x\to\pm \infty} c(x) = \lim_{z\to\pm 1} f'(z) = f'(\pm 1) <0. $$
	From this property and the continuity of $c$ we deduce that condition \eqref{eq: V1} is satisfied.
	
	Finally, since $u'$ is a $C^{1,2s-1+\gamma}$ and positive (by definition of layer) bounded solution to \eqref{linearized equation 1} we can apply Theorem~\ref{Teorema L0} to conclude the proof of the result. 
\end{proof}

\section{Odd solutions: Proof of Theorem~\ref{Mi teorema general impar} and Corollary~\ref{Nondegeneracy Ground State}} \label{sec:ProofMainResultsOdd}

In this section we prove the main results of the paper dealing with odd functions.

We begin by establishing that the quotient of an odd bounded solution and an odd bounded positive supersolution is also bounded.

\begin{proposition}{} \label{Proposicion acotacion sigma impar}
	Let $L$ be an integro-differential operator of the form \eqref{Tipo operador} with nonincreasing kernel $K$ satisfying the symmetry and ellipticity conditions \eqref{eq: K1} and \eqref{eq: K2} for some $s \in [1/2,1)$ and $0<\lambda\leq \Lambda$. Assume the potential function $c=c(x)$ is bounded, even, and satisfies condition \eqref{eq: V1} for some positive constants $R_0$ and $c_0$. 
	
	For $\alpha>2s-1$, let $w$ and $\widetilde{w}$ be two odd bounded and $C^{1}$ functions such that $[w']_{C^\alpha(\R)}$ and $[\tilde{w}']_{C^\alpha(\R)}$ are finite and satisfy
	$$ w > 0 \,\,\,\,\text{in } \,(0,R_0),$$
	$$Lw-cw \geq 0 \,\,\,\,\text{in } \,[0,+\infty),$$
	and
	$$L\widetilde{w}-c\widetilde{w} = 0 \,\,\,\,\text{in } \,[0,+\infty).$$
	
	Then, there exists a constant $C\geq 0$ such that
	$$ \left| \frac{\widetilde{w}}{w} \right| \leq C \,\,\,\,\text{in } \,\R. $$
\end{proposition}

\begin{proof}
	First, by applying Proposition~\ref{Lema principio maximo impar} and the strong maximum principle for odd functions (see Proposition 3.6 in \cite{JarohsWeth}) we get
	$$ w > 0 \,\,\,\,\text{in } \,(0,+\infty). $$
	As a consequence, the quotient $\sigma:=\tilde{w}/w$ is well-defined and continuous in $\R\setminus\{0\}$. 
	
	We will show that the quotient can be extended to be continuous and bounded in the whole real line. As in Proposition~\ref{Proposicion acotacion sigma}, this will follow after showing the positivity of the functions
	$$ \varphi_\pm = C\,w \pm \widetilde{w}, $$
	for some positive constant $C$ in $[0,+\infty)$.
	
	For this, let us take $r_0$ and $C$ such that
	$$ 0<r_0<\sqrt[2s]{\frac{\lambda}{s ||c||_{L^\infty(\R)}}}, $$
	and
	$$ C \geq \left|\left| \frac{\widetilde{w}}{w} \right|\right|_{L^\infty(r_0,R_0)}. $$
	Note that the existence of such constants is guaranteed by the boundedness of the potential function $c$ and the positivity of $w$.
	
	Now, it is enough to check that the hypotheses of Proposition~\ref{Lema principio maximo impar} are satisfied. By the choice of $C$, it is clear that $\varphi_\pm\geq 0$ in $[r_0,R_0]$ and
	$$ L \varphi_\pm -c\varphi_\pm = C \big(Lw-cw\big) \pm \big(L\widetilde{w}-c\widetilde{w}\big) =  C \big(Lw-cw\big) \geq 0 \,\,\,\,\,\text{in } \,\R^+.$$
	Furthermore, both functions $\varphi_\pm$ are odd and inherit the regularity of $w$ and $\widetilde{w}$ from being linear combinations of them.
	
	Thus, Proposition~\ref{Lema principio maximo impar} leads to
	$$ \varphi_\pm = C\,w \pm \widetilde{w} \geq 0 \,\,\,\,\,\,\, \text{ in } \,\,[0,\infty), $$
	which is equivalent to
	$$ \left| \frac{\widetilde{w}}{w} \right| \leq C \,\,\,\,\text{in } \,(0,+\infty). $$
	
	Finally, by the continuity of both $w$ and $\tilde{w}$ we can extend the result to the whole real line, concluding the proof.
\end{proof}

At this point we have all the ingredients to prove that the quotient of two odd solutions to \eqref{Ecuacion Lineal}, with one of them changing sign only once, is not only bounded but constant.

\begin{proof}[Proof of Theorem \ref{Mi teorema general impar}]
	The proof of this result is completely analogous to the one of Theorem~\ref{Mi teorema general}, applying Proposition~\ref{Proposicion acotacion sigma impar}, Corollary~\ref{Corolario integrabilidad impar}, and Corollary~\ref{Lema ecuacion sigma impar} instead of Proposition~\ref{Proposicion acotacion sigma}, Corollary~\ref{Corolario integrabilidad}, and Lemma~\ref{Lema ecuacion sigma}. 
\end{proof}

Finally, we prove Corollary~\ref{Nondegeneracy Ground State}.

\begin{proof}[Proof of Corollary \ref{Nondegeneracy Ground State}]
	First, let us point out that by the regularity theory for nonlocal equations and Proposition 1.1 in \cite{FrankLenzmann} we know that $u$ is a $C^{2,2s-1+\gamma}$ function in $\R$ for some $\gamma>0$. Furthermore, $u$ is strictly decreasing in $\R^+=(0,+\infty)$ with $u'$ being bounded. Note that the even symmetry of $u$ leads to the odd symmetry of $u'$. 

	We need to show that $u'$ is the unique bounded odd solution to
	\begin{equation} \label{linearized equation 2}
		L v - f'(u)v = 0 \,\,\,\,\,\, \text{in} \,\,\R.
	\end{equation}

	Let us take $c(x) = f'(u(x))$. It is enough to check that the hypotheses in Theorem~\ref{Mi teorema general impar} are satisfied. Since $f'\in C^\gamma([0,1])$ and $u$ is an even, continuous, and bounded function, we deduce that $c$ is even, bounded, and such that $[c]_{C^\gamma(\R)}$ is finite. Moreover, it satisfies
	$$ \lim_{x\to\pm \infty} c(x) = \lim_{z\to 0} f'(z) = f'(0) <0. $$
	Hence, \eqref{eq: V1} holds.

	Finally, since $-u'$ is a positive bounded odd solution to \eqref{linearized equation 2} in $\R^+$ we can apply Theorem~\ref{Mi teorema general impar} to complete the proof.
\end{proof}

\section*{Acknowledgments}
The author thanks Xavier Cabré for his guidance and useful discussions on the topic of this paper.

\bibliographystyle{amsplain}
\bibliography{biblio}

\end{document}